\documentclass[amscd,amssymb,verbatim]{amsart}
\usepackage{amssymb,amsfonts,amsmath,amscd}

\theoremstyle{plain}
\newtheorem{Thm}{Theorem}[section]

\newtheorem{Prop}[Thm]{Proposition}
\newtheorem{Cor}[Thm]{Corollary}

\theoremstyle{definition}
\newtheorem{Def}[Thm]{Definition}
\newtheorem{Ex}[Thm]{Example}

\theoremstyle{remark}

\newcommand{\N}{\mathbb{N}}

\newcommand{\R}{\mathbb{R}}

\newcommand{\ua}{\uparrow}
\newcommand{\XYZ}{X^{Z\ua Y}}

\begin{document}
\title[Extending Continuous Functions]{Extending Continuous Functions}
\author{Bruce Blackadar}
\address{Department of Mathematics/0084 \\ University of Nevada, Reno \\ Reno, NV 89557, USA}
\email{bruceb@unr.edu}

\date{\today}

\maketitle
\begin{abstract}
If $X$ and $Y$ are compact metrizable spaces and $Z$ a closed subspace of $Y$, the set $X^{Z\ua Y}$ of continuous functions
from $Z$ to $X$ which extend to a continuous function from $Y$ to $X$ is a subspace of the space $X^Z$ of all continuous functions from $Z$ to $X$ with the
topology of uniform convergence.  $X^{Z\ua Y}$ is neither open nor closed in $X^Z$ in general.  We examine conditions on $X$
which insure that $X^{Z\ua Y}$ is always open [closed] for any $Y$ and $Z$.
\end{abstract}

\section{Introduction}

We make the following conventions:

\paragraph{}
\begin{enumerate}
\item[]  ``Space'' will mean ``compact metrizable space.''
\item[]  ``Map'' will mean ``continuous function.''
\item[]  Unless otherwise qualified, ``subspace'' will mean ``closed subspace.''
\item[]  ``Space pair'' will mean a pair $(Y,Z)$, where $Y$ is a space and $Z$ a (closed) subspace.
\end{enumerate}

If $X$ and $Y$ are spaces, we use the standard notation $X^Y$ for the set of maps from $Y$ to $X$.  $X^Y$ becomes a topological
space under the compact-open topology, which coincides with the topology of uniform convergence (with respect to any fixed metric
on $X$, or with respect to the unique uniform structure on $X$ compatible with its topology).  $X^Y$ is metrizable.

If $(Y,Z)$ is a space pair, for any $X$ there is a natural restriction map from $X^Y$ to $X^Z$, which is continuous.
Denote by $X^{Z\ua Y}$ the range of this restriction map, i.e.\ the set of maps from $Z$ to $X$ which extend to a map from $Y$ to $X$
(we simply say such a map {\em extends to} $Y$).  Of course, $\XYZ$ is rarely all of $X^Z$; in fact, $\XYZ=X^Z$ for all space pairs
$(Y,Z)$ if and only if $X$ is an absolute retract (AR) \cite[V.2.18]{BorsukRetracts}.  Examples show that $\XYZ$ is neither open nor closed in $X^Z$ in general.
We seek conditions on $X$ insuring that $\XYZ$ is always open or closed in $X^Z$ for any space pair $(Y,Z)$.

\begin{Def}
Let $X$ be a space.
\begin{enumerate}
\item[(i)]  $X$ is {\em $e$-open} if, for every space pair $(Y,Z)$, the set $\XYZ$ is open in $X^Z$.
\item[(ii)]  $X$ is {\em $e$-closed} if, for every space pair $(Y,Z)$, the set $\XYZ$ is closed in $X^Z$.
\end{enumerate}
\end{Def}

\paragraph{}
Our main results are:

\begin{enumerate}
\item[(i)]  If $X$ is an absolute neighborhood retract (ANR), then $X$ is both $e$-open and $e$-closed (\ref{ANREop}, \ref{ANREclosed}).
\item[(ii)]  If $X$ is $e$-open, then the topology of $X$ is ``locally trivial'' (\ref{EopLocExt}, \ref{EopLocContr}).
\item[(iii)]  An $e$-open space is locally equiconnected (\ref{EopLocEqui}).
\item[(iv)]  A finite-dimensional $e$-open space is an ANR (\ref{FDEop}).
\item[(v)]  If $X$ is $e$-closed, the path components of $X$ are closed (\ref{PathComp}).
\item[(vi)]  There are $e$-closed spaces which are not $e$-open (\ref{Neclosed}).
\item[(vii)]  There are path-connected spaces which are not $e$-closed (\ref{Comb}).
\end{enumerate}

I conjecture that every $e$-open space is an ANR, and thus a space is $e$-open if and only if it is an ANR.
I do not know an example of a space which is $e$-open but not $e$-closed, and I conjecture that none exist.
I do not have a good idea how to characterize $e$-closed spaces.

\bigskip

One space we will use repeatedly is $\N^\dag$, the one-point compactification of the natural numbers $\N$ (a countable discrete space).
$\N^\dag$ will often, but not always, be identified with the subspace
$$N=\left \{ 0\right \} \cup\left \{ \frac{1}{n}:n\in\N\right \}$$
of the closed unit interval $[0,1]$.

\section{Counterexamples}

In this section, we give examples of spaces which are not $e$-open or $e$-closed.

\begin{Prop}\label{PathComp}
Let $X$ be a space.  If $X$ has a path-component which is not closed, then $X$ is not $e$-closed.
\end{Prop}

\begin{proof}
Let $P$ be a path-component of $X$ which is not closed, and $(x_n)$ a sequence in $P$ converging to $x\in X\setminus P$.
Consider the pair $([0,1],N)$.  Define maps $\phi_n$, $\phi$ from $N$ to $X$ by
$$\phi_n(1/k)=x_k\mbox{ for }k\leq n, \phi_n(1/k)=x_n\mbox{ for }k>n, \phi_n(0)=x_n$$
$$\phi(1/k)=x_k\mbox{ for all }k, \phi(0)=x\ .$$
Then $\phi_n\to\phi$ uniformly on $N$.  Each $\phi_n$ is extendible to $[0,1]$, but $\phi$ is not extendible.
\end{proof}

\begin{Ex}\label{Topsin}
Let $X$ be the ``topologist's sine curve''
$$\left \{(x,y):y=\sin\left ( \frac{\pi}{x}\right ) ,0<x\leq1\right \} \cup \left \{(0,y):-1\leq y\leq 1\right \} \subseteq\R^2\ .$$
Then $X$ has a path-component which is not closed, so $X$ is not $e$-closed.

$X$ is also not $e$-open.  Consider the pair $(Y,Z)$, where $Y=X$ and $Z=N\times\{0\}$.  Define $\phi_n$, $\phi$ from $Z$ to $X$ by
$$\phi_n((1/k,0))=(1/k,0)\mbox{ for }k\leq n, \phi_n((1/k,0))=(0,0)\mbox{ for }k>n, \phi_n((0,0))=(0,0)$$
and $\phi$ the inclusion of $Z$ into $X$.  Then $\phi_n\to\phi$ uniformly on $Z$, and $\phi$ is extendible to $Y$ (e.g.\ by the
identity map from $Y$ to $X$), but $\phi_n$ is not extendible to $Y$ for any $n$ since $(1/n,0)$ and $(1/(n+1),0)$ are in the
same path-component of $Y$ but $\phi_n((1/n,0))$ and $\phi_n((1/(n+1),0))$ are not in the same path-component of $X$.
\end{Ex}

\begin{Ex}\label{Comb}
Let $X$ be the ``infinite comb''
$$\left \{ (x,y):x\in N,0\leq y\leq1\right \} \cup \left \{ (x,0):0\leq x\leq1\right \} \subseteq\R^2\ .$$
Then $X$ is path-connected.  But $X$ is not $e$-closed.  Set $Y=[0,1]$ and $Z=N$.  Define $\phi_n$, $\phi$ from $Z$ to $X$ by
$$\phi_n(1/k)=(1/k,1)\mbox{ for }k\leq n, \phi_n(1/k)=(0,1)\mbox{ for }k>n, \phi_n(0)=(0,1)$$
$$\phi(x)=(x,1)\mbox{ for all }x\in N\ .$$
Then $\phi_n\to\phi$ uniformly on $N$, each $\phi_n$ extends in an evident way to a map from $[0,1]$ to $X$ (constant on $\left [ 0,\frac{1}{n+1}\right ]$),
but $\phi$ does not extend to $[0,1]$.

Similarly, the cone over $\N^\dag$ is not $e$-closed.
\end{Ex}

\begin{Prop}\label{Neclosed}
$\N^\dag$ is $e$-closed but not $e$-open.
\end{Prop}

\begin{proof}
To show that $\N^\dag$ is not $e$-open, let $Y=[0,1]$ and $Z=N$.  Define $\phi_n:Z\to\N^\dag$ by
$$\phi_n(1/k)=n+k\mbox{ for }k\in\N, \phi_n(0)=\infty\ .$$
Then $\phi_n$ converges uniformly to the constant function $\phi$ taking the value $\infty$.  The limit
function $\phi$ is obviously extendible to $[0,1]$ (as a constant function), but no $\phi_n$ is extendible to $[0,1]$.

To show that $\N^\dag$ is $e$-closed, first note that a continuous function from a space $Y$ to $\N^\dag$ is effectively
the same thing as a specification of a pairwise disjoint sequence of clopen sets in $Y$:  the function corresponding to
a sequence $(U_n)$ takes the value $n$ on $U_n$ and $\infty$ on the complement of the union.  If $(Y,Z)$ is a space pair,
a map $\phi:Z\to\N^\dag$ extends to $Y$ if and only if $\phi$ has the {\em clopen intersection property}: each (relatively) clopen subset $\phi^{-1}(\{n\})$ of $Z$ is the intersection
with $Z$ of a clopen set $V_n$ in $Y$.  For the $V_n$ may be converted to disjoint clopen sets
$$U_n=V_n\setminus[\cup_{k<n}V_k]$$
also with the property that $\phi^{-1}(\{n\})=Z\cap U_n$.  If $(\phi_n)$ is a sequence of maps from $Z$ to $\N^\dag$
converging uniformly on $Z$ to a map $\phi$, then for each $k\in\N$ we have $\phi^{-1}(\{k\})=\phi_n^{-1}(\{k\})$
for sufficiently large $n$, so if each $\phi_n$ is extendible, the map $\phi$ satisfies the clopen intersection
property and is thus extendible.
\end{proof}

Essentially the same argument as in the first part of this proof gives:

\begin{Cor}
If $X$ is $e$-open, then the path components of $X$ are open, i.e.\ $X$ is locally path-connected.
\end{Cor}

\begin{Ex}\label{Hawaii}
Let $X$ be the ``Hawaiian earring,'' the union over all $n$ of the circles $C_n$ in $\R^2$ of radius $\frac{1}{n}$ centered
at $\left ( \frac{1}{n},0\right )$.  Then $X$ is path-connected and locally path-connected.  But $X$ is not $e$-open.
For let $Y$ be the disk of radius 1 centered at $(1,0)$ and $Z=X\subseteq Y$, and define maps $\phi_n:Z\to X$ by mapping
the $k$'th circle to $(0,0)$ for $k\leq n$ and to itself by the identity map for $k>n$.  Then $(\phi_n)$ converges
uniformly to a constant function $\phi$.  The map $\phi$ extends to $Y$, but no $\phi_n$ extends.
\end{Ex}

\begin{Ex}\label{LocContEx}
Let $X$ be the space of \cite{BorsukEspace} (cf.\ \cite[V.11]{BorsukRetracts}), an infinite-dimensional locally contractible space which is not an ANR.
Using the notation of \cite{BorsukRetracts}, let $Y$ be the Hilbert cube and $Z=X$.  Define a sequence $(\phi_n)$ of continuous
maps from $Z$ to $X$ by
$$\phi_n(x)=\left \{ \begin{array}{ccc} {\left ( \frac{1}{n+1},x_2,x_3,\dots\right )} & \mbox{if} & {x\in\cup_{k=1}^n \dot X_k} \\
x & \mbox{otherwise} & \  \end{array} \right . $$
for $x=(x_1,x_2,x_3,\dots)$, and another map $\phi$ by $\phi((x_1,x_2,x_3,\dots))=(0,x_2,x_3,\dots)$.
Then $\phi_n\to\phi$ uniformly.  The map $\phi$ extends to $Y$ (with the same formula), but no $\phi_n$ extends to $Y$
because the identity map on $\dot X_k$ does not extend to a map from $X_k$ to $X$ for any $k$.
Thus $X$ is not $e$-open.
\end{Ex}

\section{Absolute Neighborhood Retracts}

In this section, we show that ANR's are both $e$-open and $e$-closed.

\smallskip

It will be convenient to use the following characterization of ANR's (cf.\ \cite{BlackadarShape}).  The proof that this condition
is equivalent to the local Tietze extension property is a simple compactness exercise.

\begin{Prop}\label{ANRChar}
Let $X$ be a space.  Then $X$ is an ANR if and only if, whenever $Y$ is a space, $(Z_n)$ a decreasing sequence
of subspaces of $Y$ with $Z=\cap_n Z_n$, and $\phi$ is a map from $Z$ to $X$, then $\phi$ extends to a map from
$Z_n$ to $X$ for sufficiently large $n$.
\end{Prop}

The fact that ANR's are $e$-open is a simple consequence of \ref{ANRChar} using a standard argument:

\begin{Prop}\label{ANREop}
Let $X$ be an ANR.  Then $X$ is $e$-open.
\end{Prop}

\begin{proof}
Let $(Y,Z)$ be a space pair, and $\phi_n$, $\phi$ in $X^Z$, with $\phi_n\to\phi$ uniformly on $Z$ and $\phi\in\XYZ$.
We must show $\phi_n\in\XYZ$ for sufficiently large $n$.

Let $\tilde Y=N\times Y$,
$$\tilde Z_n=\left ( \left ( N\cap\left [ 0,\frac{1}{n}\right ] \right ) \times Y\right ) \cup \left ( N\times Z\right )$$
$$\tilde Z=\left ( \{0\}\times Y\right ) \cup \left ( N\times Z\right ) \ .$$
Then $\tilde Z=\cap_n \tilde Z_n$.  Define a function $\psi:\tilde Z\to X$ by
$$\psi((1/n,z))=\phi_n(z),\ \psi((0,y)=\bar\phi(y)$$
where $\bar\phi$ is an extension of $\phi$ to $Y$.  Since $\phi_n\to\phi$ uniformly on $Z$, $\psi$ is continuous.
By \ref{ANRChar}, $\psi$ extends to $\tilde Z_n$ for sufficiently large $n$.  Restricting this extension to
$\{1/n\}\times Y$ gives an extension of $\phi_n$ for large $n$.
\end{proof}

In fact, this proof shows more: If $\bar\phi$ is an extension of $\phi$ to $Y$, then the extensions $\bar\phi_n$
of $\phi_n$ to $Y$ for sufficiently large $n$ can be chosen so that $\bar\phi_n\to\bar\phi$ uniformly on $Y$.

\bigskip

A very important property of ANR's is that sufficiently close maps from a space $Y$ to an ANR $X$ are homotopic (cf.\ \cite[IV.1.1]{HuRetracts}).
Recall that a
homotopy between two elements of $X^Y$ is by definition an element of $X^{Y\times[0,1]}$, but this space is
naturally homeomorphic to $(X^Y)^{[0,1]}$ (), i.e.\ a homotopy from $\phi_0$ to $\phi_1$ is just a continuous path
$(\phi_t)$ in $X^Y$ from $\phi_0$ to $\phi_1$.

\begin{Thm}\label{CloseHomotopic}
Let $X$ be an ANR.  Fix a metric $\rho$ on $X$.  Then there is an $\epsilon>0$ such that, whenever $Y$ is a space
and $\phi_0$, $\phi_1$ maps of $Y$ into $X$ with $\rho(\phi_0(y),\phi_1(y))<\epsilon$ for all $y\in Y$, then
$\phi_0$ and $\phi_1$ are homotopic.  (The $\epsilon$ depends on $X$ and $\rho$, but not on $Y$.)
\end{Thm}



We can prove a slightly weaker statement for general $e$-open spaces:

\begin{Thm}\label{EopHomotopy}
Let $X$ be an $e$-open space, $Y$ a space, and $\phi_n,\phi\in X^Y$ with $\phi_n\to\phi$ uniformly on $Y$.
Then $\phi_n$ is homotopic to $\phi$ for sufficiently large $n$.  In fact, there is a continuous path $\psi_t$,
$0\leq t\leq\frac{1}{n}$ for some $n$, such that $\psi_{1/k}=\phi_k$ for all $k\geq n$ and $\psi_0=\phi$.
\end{Thm}

\begin{proof}
Let $\tilde Y=[0,1]\times Y$, and $\tilde Z=N\times Y\subseteq \tilde Y$.  Define a sequence $\tilde\phi_n$ of
functions from $\tilde Z$ to $X$ by
$$\tilde\phi_n((1/k,y))=\phi_k(y)\mbox{ for }k>n,\ \tilde\phi_n(1/k,y)=\phi(y)\mbox{ for }k\leq n,\ \phi_n((0,y))=\phi(y)$$
and define $\tilde\phi:\tilde Z\to X$ by $\tilde\phi((x,y))=\phi(y)$ for all $x\in N$.  Then $(\tilde\phi_n)$ is
a sequence in $X^{\tilde Z}$ converging uniformly to $\tilde\phi$, and $\tilde\phi$ extends to $\bar\phi:\tilde Y\to X$
defined by $\bar\phi((t,y))=\phi(y)$ for all $t\in[0,1]$, $y\in Y$.  Thus, since $X$ is $e$-open, $\tilde\phi_n$
extends to $\tilde Y$ for sufficiently large $n$, defining the desired homotopy.
\end{proof}

A better result will be shown later (\ref{EopLocEqui}).

Showing that an ANR is $e$-closed is somewhat more complicated than the proof that it is $e$-open:

\begin{Thm}\label{ANREclosed}
Every ANR is $e$-closed.
\end{Thm}

\begin{proof}
Let $X$ be an ANR, fix a metric $\rho$ on $X$, and let $\epsilon>0$ be the number given by \ref{CloseHomotopic}.
Let $(Y,Z)$ be a space pair, and $(\phi_n)$ a sequence in $X^Z$ converging uniformly on $Z$ to $\phi:Z\to X$.
Suppose each $\phi_n$ extends to $Y$.  We must show that $\phi$ also extends to $Y$.

Since $X$ is an ANR, $\phi$ extends to a neighborhood $U$ of $Z$.  Choose $n$ so that $\rho(\phi_n(z),\phi(z))<\frac{\epsilon}{2}$
for all $z\in Z$.  Let $\tilde\phi$ be an extension of $\phi$ to $U$, and $\bar\phi_n$ an extension of $\phi_n$ to $Y$.
Then there is an open neighborhood $V$ of $Z$ contained in $U$ such that $\rho(\bar\phi_n(y),\tilde\phi(y))<\epsilon$ for all $y\in\bar V$.
Thus, by \ref{CloseHomotopic}, $\bar\phi_n$ and $\tilde\phi$ are homotopic as maps from $\bar V$ to $X$.
Let $(\psi_t)$ be a homotopy with $\psi_0=\bar\phi_n$ and $\psi_1=\tilde\phi$ on $\bar V$.

Since $Y$ is normal, there is a map $f:Y\to[0,1]$ with $f\equiv 1$ on $Z$ and $f\equiv 0$ on $Y\setminus V$.
Define $\bar\phi:Y\to X$ by
$$\bar\phi(y)=\left \{ \begin{array}{ccc} {\psi_{f(y)}(y)} & \mbox{if} & {y\in\bar V} \\ \, & \, & \, \\ \bar\phi_n(y) & \mbox{if} & {y\notin\bar V} \end{array} \right . \ .$$
It is easy to check that $\bar\phi$ is continuous and extends $\phi$ to $Y$.
\end{proof}

\section{Local Extendibility}

In this section, we consider the converse of \ref{ANREop}: is an $e$-open space necessarily an ANR?
The next local extendibility result comes close to showing this.

\begin{Thm}\label{EopLocExt}
Let $X$ be an $e$-open space.  Fix a metric $\rho$ on $X$.  Then for any $\epsilon>0$ there is a $\delta>0$
such that, whenever $(Y,Z)$ is a space pair and $\phi$ is a map from $Z$ to $X$ with $diam(\phi(Z))<\delta$,
$\phi$ extends to a map $\bar\phi:Y\to X$ with $diam(\bar\phi(Y))<\epsilon$.  (The $\delta$ depends only
on $X$, $\rho$, and $\epsilon$, not on $Y$ or $Z$.)
\end{Thm}

\begin{proof}
Suppose the result is false.  Then there is an $\epsilon>0$ such that, for every $n$, there is a space pair
$(Y_n,Z_n)$ and a map $\phi_n:Z_n\to X$ with $diam(\phi_n(Z_n))<\frac{1}{n}$ but such that $\phi_n$ has no
extension $\bar\phi_n$ to $Y_n$ with $diam(\bar\phi_n(Y_n))<\epsilon$.

Let $z_n\in Z_n$.  Passing to a subsequence, we may assume the sequence $(\phi_n(z_n))$ converges in $X$
to a point $p$.  Since $diam(\phi_n(Z_n))\to0$, if $w_n$ is any point of $Z_n$ for each $n$ we have
$\phi_n(w_n)\to p$.

Let $Y$ be the one-point compactification of the disjoint union of the $Y_n$, and regard $Y_n$ as a subspace of $Y$.
Let $Z$ be the union of the $Z_n$ and the point at infinity in $Y$.  Define maps $\tilde\phi_n$ from $Z$ to $X$ as follows:
$$\tilde\phi_n(z)=\left \{ \begin{array}{ccc} \phi_k(z) & \mbox{ if } & {z\in Z_k,\ k>n} \\ p & \mbox{ if } & {z\in Z_k,\ k\leq n} \\
p & \mbox{ if } & z=\infty \end{array} \right . \ .$$
Then $(\tilde\phi_n)$ converges uniformly on $Z$ to the constant function with value $p$.  Since this function is
extendible to $Y$ (e.g.\ to a constant function), $\tilde\phi_n$ extends to $Y$ for some $n$ (any sufficiently large $n$), giving extensions
$\bar\phi_k$ of $\phi_k$ to $Y_k$ for all  $k>n$.  If $y_k$ is any point of $Y_k$ for each $k>n$,
then $y_k\to\infty$ in $Y$, so $\bar\phi_k(y_k)\to p$ for any such sequence.  It follows that $diam(\bar\phi_k(Y_k))\to0$,
a contradiction.
\end{proof}

The result may be rephrased as a ``local AR'' property:

\begin{Cor}\label{EopLocalAR}
Let $X$ be an $e$-open space, and $p\in X$.  For every neighborhood $U$ of $p$ there is a neighborhood $V$ of $p$
such that, whenever $(Y,Z)$ is a space pair and $\phi:Z\to V$ is a map, then $\phi$ extends to a map $\bar\phi:Y\to U$.
\end{Cor}

In particular:

\begin{Cor}\label{EopLocContr}
Let $X$ be an $e$-open space. Then $X$ is locally contractible.  In fact, if $p\in X$ and $U$ is any neighborhood
of $p$ in $X$, then there is a neighborhood $V$ of $p$ which can be contracted to $p$ relative to $p$ within $U$.
\end{Cor}

\begin{proof}
Fix a metric $\rho$ on $X$, and an $\epsilon>0$ such that the open ball $B_\epsilon(p)$ of radius $\epsilon$ around $p$
is contained in $U$.  Let $\delta$ be as in \ref{EopLocExt}, and let $V$ be a neighborhood of $p$ of diameter $<\delta$
(e.g.\ $V=B_{\delta/3}(p)$).  Let $Y$ be the cone over $\bar V$, i.e.\
$$Y=\bar V\times[0,1]/\sim$$
where $\sim$ identifies $\bar V\times\{1\}$ to a single point,
and let $Z\subseteq Y$ be the ``spiked base'' over $p$:
$$Z=\left ( \bar V\times\{0\}\right ) \cup \left ( \{p\}\times [0,1]\right ) \ .$$
There is an obvious map of $Z$ into $X$ which is the identity map on the base and sends the spike to $p$.
By \ref{EopLocExt} this map can be extended to a map from $Y$ to $U$, giving the desired homotopy.
\end{proof}

Combining this with \cite[V.10.4]{BorsukRetracts}, we obtain:

\begin{Cor}\label{FDEop}
A finite-dimensional space which is $e$-open is an ANR.
\end{Cor}

Example \ref{LocContEx} provides evidence that the finite-dimensionality assumption is not necessary.  In fact,
\ref{EopLocExt} comes quite close to showing the Lefschetz condition (cf.\ \cite[V.8]{BorsukRetracts}) which characterizes ANR's: roughly,
$X$ satisfies the Lefschetz condition if, whenever $K$ is a polyhedron and $\phi$ a map of the vertices of $K$
into $X$ with the images of adjacent vertices sufficiently close together, then $\phi$ extends to a map from $K$ to $X$.
\ref{EopLocExt} can be applied successively to extend $\phi$ to the $k$-skeleton for each $k$, showing that the
Lefschetz condition is satisfied for each $n$ for polyhedra of dimension $\leq n$; however, the $\delta$
needed will depend on $n$ and not just on $\epsilon$, so the full Lefschetz condition is not obtained from \ref{EopLocExt}.

\bigskip

There are various alternative characterizations of local equiconnectedness (cf.\ \cite{FoxFibre2}, \cite{DugundjiLocal}, \cite{HimmelbergTheorems}).
Using one of these, we can also obtain:

\begin{Cor}\label{EopLocEqui}
An $e$-open space is locally equiconnected.
\end{Cor}

\begin{proof}
Let $X$ be an $e$-open space.  We will show that for every $\epsilon>0$ there is a $\delta>0$ such that, if $Y$ is a space
and $\phi_0$, $\phi_1$ maps from $Y$ to $X$ which are $\delta$-close, and $Z=\{z\in Y:\phi_0(z)=\phi_1(z)\}$, then there
is a homotopy $(\phi_t)$ of maps from $Y$ to $X$ ($0\leq t\leq1$) with $\phi_t$ and $\phi_s$ $\epsilon$-close for all $t,s\in[0,1]$,
and such that $\phi_t(z)=\phi_0(z)=\phi_1(z)$ for all $z\in Z$ (cf.\ \cite[2.5]{DugundjiLocal}).

Let $\epsilon>0$, and let $\delta>0$ be as in \ref{EopLocExt}.  Let $\tilde Y=Y\times[0,1]$,
$$\tilde Z=(Y\times\{0,1\} )\cup(Z\times[0,1])\subseteq\tilde Y \ .$$
Define $\tilde\phi,\tilde\psi:\tilde Z\to X$ by $\tilde\phi(y,0)=\phi_0(y)$, $\tilde\phi(y,1)=\phi_1(y)$, $\tilde\phi(z,t)=\phi_0(z)=\phi_1(z)$;
$\tilde\psi(y,0)=\tilde\psi(y,1)=\phi_0(y)$, $\tilde\psi(z,t)=\phi_0(z)$.  The $\tilde\phi$ and $\tilde\psi$ are $\delta$-close
maps from $\tilde Z$ to $X$, and $\tilde\psi$ extends to $\tilde Y$ (setting $\tilde\psi(y,t)=\phi_0(y)$ for all $y,t$), and hence
$\tilde\phi$ also extends to $\tilde Y$, giving the desired homotopy.
\end{proof}

It is a longstanding open problem whether a locally equiconnected space is an ANR; we have shown that the class of $e$-open
spaces lies in between.  I am indebted to Sergey Melikhov (via MathOverflow) for pointing out the relevance of local
equiconnectedness.  I also thank Jan van Mill for helpful comments.

\section{Noncommutative Versions}

My primary interest in the problems discussed in this article is in the noncommutative versions of the problems
for C*-algebras and their relation to semiprojectivity.  See the companion article \cite{BlackadarHomotopy} for a discussion.

\bibliography{funextref}
\bibliographystyle{alpha}

\end{document}